\documentclass[11pt,reqno]{amsart}
\usepackage{lmodern}
\usepackage{amsmath,amssymb,amsthm,mathtools,mathrsfs}
\usepackage[margin=2.5cm]{geometry}
\usepackage{microtype}
\usepackage{needspace}
\usepackage{xcolor}
\usepackage{hyperref}
\usepackage[nameinlink,capitalize,noabbrev]{cleveref}

\hypersetup{
  colorlinks=true,
  linkcolor=blue!45!black,
  citecolor=green!35!black,
  urlcolor=blue!55!black,
  pdftitle={Global regularity for the unidirectional Euler-alignment system with supercritical dissipation},
  pdfauthor={Changhui Tan and Liutang Xue},
  pdfsubject={Global regularity for a supercritical Euler--alignment system},
  pdfkeywords={Euler--alignment system, supercritical dissipation, fractional Laplacian, moduli of continuity, scaling invariance, weighted convexity}
}

\renewcommand{\eqref}[1]{\textup{(\ref{#1})}}
\numberwithin{equation}{section}

\newtheorem{theorem}{Theorem}[section]
\newtheorem{proposition}[theorem]{Proposition}
\newtheorem{lemma}[theorem]{Lemma}

\theoremstyle{definition}

\newcommand{\T}{\mathbb T}
\newcommand{\R}{\mathbb R}
\newcommand{\N}{\mathbb N}
\newcommand{\Z}{\mathbb Z}
\newcommand{\dd}{\,\mathrm d}

\newcommand{\mean}[1]{\langle #1\rangle}
\newcommand{\norm}[2]{\left\lVert #1\right\rVert_{#2}}

\newcommand{\comm}[2]{\left[#1,#2\right]}

\def\G{\Gamma}
\def\rhom{\underline\rho}

\title[Global regularity for supercritical Euler-alignment]{Global regularity for the unidirectional Euler-alignment system with supercritical dissipation}

\author[Changhui Tan]{Changhui Tan}
\address[Changhui Tan]{\newline Department of Mathematics, University of South Carolina, Columbia SC 29208, USA}
\email{tan@math.sc.edu}

\author[Liutang Xue]{Liutang Xue}
\address[Liutang Xue]{\newline School of Mathematical Sciences, Laboratory of Mathematics and Complex Systems (MOE), Beijing Normal University, Beijing 100875, P.R. China}
\email{xuelt@bnu.edu.cn}

\date{\today}


\thanks{\textit{Acknowledgment.}
C. Tan is supported by NSF grant DMS-2238219.
L. Xue is supported by National Natural Science Foundation of China (No. 12271045, 12571245, 12671270).}

\begin{document}

\begin{abstract}
We prove global well-posedness of the periodic unidirectional
Euler--alignment system in every dimension for supercritical
dissipation of order $0<\alpha<1$ and arbitrary non-vacuum initial
data in the Sobolev class of the local theory.
The key idea is to propagate simultaneously two moduli of continuity
at critical scales for the density $\rho$ and the potential gradient
$\Gamma=\nabla\Lambda^{-\alpha}u$, both of which have scaling-invariant
amplitudes.  We also establish exponential alignment of the velocity
and exponential convergence of the density to a traveling profile.
\end{abstract}

\maketitle

\section{Introduction}

The Euler--alignment system is a macroscopic model of collective motion
in which individuals adjust their velocities toward those of their
neighbors.  It arises from the Cucker--Smale particle dynamics
\cite{cucker2007emergent}; see also \cite{ha2008particle,carrillo2017review}
for hydrodynamic derivations and related models.  The system takes the form
\begin{equation}\label{eq:EAS}
 \begin{cases}
   \,\partial_t\rho+\nabla\cdot(\rho\mathbf u) = 0,\\[5pt]
   \,\partial_t\mathbf u+\mathbf u\cdot\nabla\mathbf u
   = \displaystyle\int_{\R^d} \psi(x-y)\bigl(\mathbf u(y)-\mathbf u(x)\bigr) \rho(y)\dd y,
 \end{cases}
\end{equation}
where $\rho$ is the density, $\mathbf u=(u_1,\ldots,u_d)$ is the velocity,
and the nonnegative communication kernel $\psi$ describes the strength
of the alignment interaction.

The regularity of solutions to \eqref{eq:EAS} depends strongly on the
singularity of $\psi$.  For bounded Lipschitz kernels, alignment acts
primarily as a nonlocal damping mechanism, and global regularity is
governed by critical-threshold conditions on the initial data; see,
for example, \cite{tadmor2014critical, carrillo2016critical}.
Weakly singular but integrable kernels exhibit related threshold
behavior \cite{tan2020euler}.

In this paper, we consider strongly singular kernels of the form
\[
  \psi(x)=c_{d,\alpha}|x|^{-d-\alpha},
  \qquad x\in\R^d\setminus\{0\},\quad 0<\alpha<2,
\]
where $c_{d,\alpha}>0$ is normalized so that the alignment force becomes
\[
  -\Lambda^\alpha(\rho\mathbf u)+\mathbf u\Lambda^\alpha\rho,
  \qquad \Lambda^\alpha=(-\Delta)^{\alpha/2}.
\]
Singular integrals are understood in the principal-value sense when necessary. The alignment force thus generates a density-weighted nonlocal diffusion, whose leading-order dissipative effect remains nondegenerate as long as the density stays strictly positive. In contrast, the presence of vacuum may lead to singularity formation even in one dimension \cite{tan2019singularity,arnaiz2021singularity}. Motivated by this distinction, throughout this paper we consider \eqref{eq:EAS} on the periodic domain $\T^d$ and and assume that the initial density is bounded away from vacuum $\rho_0(x)>0$. For periodic solutions, the singular integral operators above are understood as acting on periodic extensions.

The one-dimensional theory for the singular Euler--alignment system has been established in \cite{do2018global,shvydkoy2017eulerian,shvydkoy2018eulerian}. A key feature is an additional structure that is not apparent from the velocity equation alone. The auxiliary quantity
\[
  G=\partial_xu-\Lambda^\alpha\rho
\]
satisfies the continuity equation
\[
  \partial_tG+\partial_x(uG)=0.
\]
This remarkable conservation law provides an essential structural control on the coupling between the velocity and the density.

The dissipative mechanism is particularly transparent in the special class $G_0=0$. In this case $G\equiv0$, and the system reduces to
\[
  \partial_t\rho+u\partial_x\rho
  =-\rho\Lambda^\alpha\rho,
  \qquad
  u=-\partial_x\Lambda^{\alpha-2}\rho.
\]
Thus the alignment interaction appears directly as a density-weighted fractional dissipation, closely related to fractional porous-medium flow with nonlocal pressure \cite{caffarelli2011nonlinear}. More generally, the conservation law for $G$, together with this dissipative mechanism, is a fundamental ingredient in the global well-posedness theory for arbitrary smooth non-vacuum initial data.

The theory is much less developed in multiple dimensions. For a general vector field $\mathbf u$, the natural analogue
\[
  G=\nabla\cdot\mathbf u-\Lambda^\alpha\rho
\]
obeys
\[
  \partial_tG+\nabla\cdot(G\mathbf u) = (\nabla\cdot\mathbf u)^2 -\operatorname{tr}\bigl((\nabla\mathbf u)^2\bigr).
\]
The source term on the right is known as the spectral gap. It has no definite sign and does not vanish for a general multidimensional flow. Thus the conservation law for $G$, which is decisive in one dimension, is lost.

The absence of such a conserved quantity presents a fundamental difficulty in controlling the coupling between the density and the velocity gradient. In particular, the one-dimensional mechanism that leads to global regularity for arbitrary smooth non-vacuum data does not extend directly to general multidimensional flows. Consequently, the existing global well-posedness theory in multiple dimensions is largely perturbative. Global solutions have been constructed for small perturbations of aligned flocks or constant states; see, for example, \cite{shvydkoy2019global,danchin2019regular}. A global theory for arbitrary smooth non-vacuum initial data remains largely open.

Unidirectional flows form an important intermediate class between the one-dimensional and general multidimensional settings. Lear and Shvydkoy \cite{lear2021unidirectional} considered velocities of the form
\[
  \mathbf u(x,t)=u(x,t)\mathbf d,
  \qquad \mathbf d\in\mathbb S^{d-1},
\]
an ansatz that is preserved by the Euler--alignment evolution. Without loss of generality, we take the distinguished direction to be $\mathbf d=e_1$. Then
\[
  \mathbf u=(u,0,\ldots,0),
  \qquad
  \nabla\cdot\mathbf u=\partial_1u.
\]
For such flows, the spectral gap vanishes identically:
$
  (\nabla\cdot\mathbf u)^2
  -\operatorname{tr}\bigl((\nabla\mathbf u)^2\bigr)=0.
$
Hence the auxiliary quantity
\[
  G=\partial_1u-\Lambda^\alpha\rho
\]
satisfies the conservation law
\[
  \partial_tG+\partial_1(uG)=0.
\]
Thus the key one-dimensional conservation structure is recovered.
The unidirectional Euler--alignment system therefore takes the form
\begin{equation}\label{eq:main}
 \begin{cases}
   \,\partial_t\rho+\partial_1(\rho u) = 0,\\
   \,\partial_tu+u\partial_1u
   =-\Lambda^\alpha(\rho u)+u\Lambda^\alpha\rho,
 \end{cases}
\end{equation}
on $\T^d\times[0,T)$. Although the transport is directed along $x_1$, both $\rho$ and $u$ depend on all $d$ spatial variables, and $\Lambda^\alpha$ acts in every direction. Accordingly, this is not merely a one-dimensional problem embedded in a higher-dimensional space. In particular, the relation
\begin{equation}\label{eq:relation}
  \partial_1u=\Lambda^\alpha\rho+G
\end{equation}
controls only the derivative of $u$ in the transport direction and provides no direct control of the transverse derivatives $\partial_ku$, $k\neq1$. Even in the special class $G_0=0$, where $G\equiv0$, this relation alone does not close the transverse regularity estimates. This constitutes the main structural distinction between the unidirectional system and its genuinely one-dimensional counterpart.

The fractional Burgers equation is a second useful benchmark.  If the
density $\rho$ in the velocity equation is frozen at one, the
alignment operator becomes $-\Lambda^\alpha u$, leading formally to
\[
  \partial_tu+u\partial_1u=-\Lambda^\alpha u.
\]
Solutions of fractional Burgers remain globally regular for
$\alpha\geq1$, while finite-time singularities may occur for
$0<\alpha<1$ \cite{kiselev2008blow}.  In the
one-dimensional Euler--alignment system, the variable density weight
enhances the dissipative mechanism sufficiently to prevent this
Burgers-type breakdown for non-vacuum data.  A central question left by
the preceding theory was whether this enhancement survives in the
multidimensional unidirectional setting, or whether it acts
effectively only in the distinguished direction $x_1$.

Lear and Shvydkoy \cite{lear2021unidirectional} established global regularity for the unidirectional system in the subcritical range $1<\alpha<2$. The terminology is consistent with the natural scaling
\begin{equation}\label{eq:scaling}
  \rho_\lambda(x,t)=\rho(\lambda x,\lambda^\alpha t),
  \qquad
  u_\lambda(x,t)=\lambda^{\alpha-1}
  u(\lambda x,\lambda^\alpha t).
\end{equation}
When $\alpha>1$, the velocity amplitude is subcritical under this scaling. In particular, the fractional dissipation is strong enough to dominate the Burgers-type transport, much as in the subcritical fractional Burgers equation. In this regime, one may therefore treat the velocity equation directly as a dissipative transport equation, without relying on the more delicate one-dimensional mechanism. Once $u$ is controlled, the density can be recovered from
\[
  \rho=\partial_1\Lambda^{-\alpha}u-\Lambda^{-\alpha}G.
\]
Indeed, for $\alpha>1$, the operator $\partial_1\Lambda^{-\alpha}$ has negative order $1-\alpha<0$ and is therefore smoothing. Together with the control of the conserved quantity $G$, this allows the regularity of $\rho$ to be recovered from that of $u$.

The critical case $\alpha=1$ is considerably more subtle. Lear \cite{lear2023global} obtained global regularity under the special null-entropy condition $G_0=0$, as well as a separate small-data result for general initial entropy. The first large-data global well-posedness result for general unidirectional flows at the critical exponent was obtained by Li et al. \cite{li2024global}. In the critical regime, both unknowns have scaling-invariant amplitudes,
\[
  \rho_\lambda(x,t)=\rho(\lambda x,\lambda t),
  \qquad
  u_\lambda(x,t)=u(\lambda x,\lambda t),
\]
and the reconstruction of the density becomes borderline:
\[
  \rho=\partial_1\Lambda^{-1}u-\Lambda^{-1}G.
\]
The leading operator $\partial_1\Lambda^{-1}$ is a Riesz transform and is not bounded from $L^\infty$ to $L^\infty$. Thus control of the velocity alone no longer yields sufficient control of the density. This difficulty also appears directly in the alignment force: the commutator remainder contains simultaneous increments of $\rho$ and $u$ and cannot be controlled by a modulus of continuity for $u$ alone.

The main idea of \cite{li2024global} is therefore to propagate moduli of continuity for $\rho$ and $u$ simultaneously. Since both quantities have critical, scaling-invariant amplitudes at $\alpha=1$, the two moduli can be chosen compatibly so that their breakthrough scenarios are coupled and ruled out together. This simultaneous propagation of multiple moduli of continuity is a genuinely system-level extension of the modulus method developed for critical scalar equations such as the critical Burgers and quasi-geostrophic equations \cite{kiselev2008blow,kiselev2007global}.

The supercritical regime $0<\alpha<1$ presents a fundamentally different difficulty. Here the fractional dissipation is too weak, by itself, to control Burgers-type steepening. Thus any large-data global regularity mechanism must exploit the special structure inherited from the one-dimensional Euler--alignment system. However, in the multidimensional unidirectional setting, the relation \eqref{eq:relation} controls only the derivative in the distinguished direction $x_1$. It provides no direct control of the transverse derivatives of $u$, while in the supercritical range these transverse derivatives cannot be handled by dissipation alone.

In \cite{li2024global}, this difficulty led to the continuation criterion
\begin{equation}\label{eq:contcrit}
  u\in L^\infty_t C^\sigma_x,
  \qquad \text{for some }\sigma\in(1-\alpha,1).
\end{equation}
In view of the scaling \eqref{eq:scaling}, the exponent $1-\alpha$ is critical for the velocity. Thus, once a H\"older bound strictly above this critical level is available, the remaining nonlinear effects become perturbative and the solution can be continued. The essential unresolved issue, however, is to prove such a bound dynamically from arbitrary large initial data; this was not addressed in \cite{li2024global}.
\bigskip

The goal of the present work is to resolve precisely this issue. We prove global well-posedness in the full supercritical range without imposing any additional continuation criterion or a priori H\"older control on the velocity. In particular, the transverse regularity is obtained dynamically, despite the weakness of the dissipation and the fact that the structural relation acts directly only in the distinguished direction. We also prove exponential alignment of the velocity and convergence of the density to a traveling profile.

Our main result is stated as follows.
\begin{theorem}[Global classical solutions and asymptotic alignment]
\label{thm:main}
Let $d\geq1$, $0<\alpha<1$, and let
$m\in\mathbb N$ satisfy $m>d/2+1$.  Assume that
\begin{equation}
  (\rho_0,u_0)
  \in
  H^{m+\alpha}(\T^d)\times H^{m+1}(\T^d),
  \qquad
  \inf_{x\in\T^d}\rho_0(x)>0.
  \label{eq:data}
\end{equation}
Then the unidirectional Euler--alignment system
\eqref{eq:main} admits a unique global
classical solution satisfying
\begin{align}
  \rho
  &\in
  C_{\mathrm w}\bigl([0,\infty);H^{m+\alpha}(\T^d)\bigr),
  \notag\\
  u
  &\in
  C_{\mathrm w}\bigl([0,\infty);H^{m+1}(\T^d)\bigr)
  \cap
  L^2_{\mathrm{loc}}
  \bigl([0,\infty);\dot H^{m+1+\alpha/2}(\T^d)\bigr).
  \label{eq:global-solution-class}
\end{align}
Here $C_{\mathrm w}$ denotes weak continuity in time.

Moreover, there exist a positive profile $\rho_\infty\in W^{1,\infty}(\T^d)$
with the same total mass as $\rho_0$ and constants $C<\infty$ and
$\kappa>0$, depending only on the initial data and the fixed parameters,
such that
\begin{align}
  \norm{\rho(\cdot,t)-\rho_\infty(\cdot-\bar u t e_1)}{L^\infty(\T^d)}
  &\leq Ce^{-\kappa t},
  \label{eq:main-density-alignment}\\
  \norm{u(\cdot,t)-\bar u}{W^{1,\infty}(\T^d)}
  &\leq Ce^{-\kappa t},\qquad t\geq0,
  \label{eq:main-velocity-alignment}
\end{align}
where the density-weighted average velocity is
\begin{equation*}
  \bar u
  :=\frac{\int_{\T^d}\rho_0(x)u_0(x)\dd x}
          {\int_{\T^d}\rho_0(x)\dd x}.
\end{equation*}
\end{theorem}

Let us describe the main idea and strategy of the proof. The key is to introduce the potential gradient
\[
  \G:=\nabla\Lambda^{-\alpha}u,
\]
where $\Lambda^{-\alpha}$ annihilates the zero Fourier mode. There are two important reasons for this choice. First, $\G$ has scaling-invariant amplitude. Indeed, under the scaling \eqref{eq:scaling},
\[
  \rho_\lambda(x,t)=\rho(\lambda x,\lambda^\alpha t),
  \qquad
  \G_\lambda(x,t)=\G(\lambda x,\lambda^\alpha t).
\]
Thus $\rho$ and $\G$ live at the same critical scale, making them natural quantities for a simultaneous modulus-of-continuity argument. Second, $\G$ retains the crucial Euler--alignment structure encoded in \eqref{eq:relation}. Applying $\nabla\Lambda^{-\alpha}$ to the velocity equation and using
\[
  \partial_1u=\Lambda^\alpha\rho+G,
\]
one finds a cancellation of all derivatives of the density and obtains the transport-diffusion equation
\begin{equation}
  \partial_t\G+u\partial_1\G+\rho\Lambda^\alpha\G=f,
  \qquad
  f:=-\comm{\nabla\Lambda^{-\alpha}}{u}G.
  \label{eq:H-intro}
\end{equation}
This cancellation is essential: the leading part of the $\G$-equation is a density-weighted fractional diffusion, while the remaining forcing has the commutator structure above.

Motivated by the effectiveness of simultaneously propagating moduli of continuity in the critical case \cite{li2024global}, we seek to propagate compatible moduli for $\rho$ and $\G$. The main new difficulty is the commutator forcing $f$. Its control is carried out in two steps. 

We first establish an $L^\infty$ bound for $\G$, which is needed to initiate the modulus argument. Since the commutator prevents a direct application of the maximum principle, we begin with the endpoint commutator estimate
\[
  \norm{f}{L^p}
  \leq C_{d,\alpha,p}
  \norm{G}{L^\infty}\norm{\G}{L^p},
  \qquad 1<p<\infty.
\]
Together with the uniform bound on $G$ and the C\'ordoba--C\'ordoba inequality
\cite{cordoba2003pointwise}, this closes the finite-$p$ estimate without requiring any control of density gradients. Since the commutator estimate does not extend directly to $p=\infty$, we derive an $L^p\to L^\infty$ smoothing estimate for the corresponding homogeneous equation by fractional Moser iteration. Duhamel's formula then yields an $L^\infty$ bound for $\G$. Combining velocity alignment with a nonlinear dissipation estimate further gives exponential decay of $\norm{\G(t)}{L^\infty}$.

We then propagate a fixed modulus for $\rho$ and an exponentially decreasing modulus for $\G$,
\[
  \omega_1(\xi)=\omega_\lambda^\delta(\xi),
  \qquad
  \omega_2(\xi,t)=e^{-\kappa t}\omega_\lambda^\delta(\xi),
\]
where $\omega_\lambda^\delta$ is the family defined in
\eqref{eq:MOC-family}. The common critical scaling of $\rho$ and $\G$ is crucial here: it allows us first to choose the amplitude $\delta$ sufficiently small so that the transport and variable-coefficient contributions can be absorbed by the dissipation, and then to choose the spatial scale $\lambda$ sufficiently small to accommodate arbitrary initial data. The main work in the $\G$ breakthrough scenario is to estimate the difference
\[
  f(x,t)-f(y,t)
\]
at a possible contact pair. By exploiting both the commutator structure and the simultaneous moduli for $\rho$ and $\G$, we show that this difference has a modulus compatible with the dissipative term. This allows the dissipation to rule out both the density and $\G$ breakthrough scenarios.

The propagated moduli yield
\[
  \sup_{t\geq0}\norm{\nabla\rho(t)}{L^\infty}<\infty,
  \qquad
  [\G(t)]_{\mathrm{Lip}}\lesssim e^{-\kappa t},
\]
and, through the reconstruction of $u$ from $\G$, also
\[
  \norm{\nabla u(t)}{L^\infty}\lesssim e^{-\kappa t}.
\]
This is substantially stronger than the conditional H\"older control required in \eqref{eq:contcrit}: rather than assuming that $u$ remains in $C^\sigma$ for some $\sigma>1-\alpha$, the argument dynamically produces a global Lipschitz bound for the velocity. These estimates close the continuation criterion and yield global regularity.

\bigskip
The rest of the paper is organized as follows. \cref{sec:setup} collects the local well-posedness theory and continuation criterion, together with the basic a priori bounds and the endpoint commutator estimate. \cref{sec:Gamma-bounds} derives the evolution equation for the potential gradient $\G$ and establishes its $L^p$ and $L^\infty$ bounds, followed by exponential decay. \cref{sec:MOC-Gamma} develops the simultaneous modulus-of-continuity argument for $\rho$ and $\G$, including the breakthrough analysis and the key estimate on the commutator forcing. \cref{sec:global} closes the continuation argument and proves exponential alignment of the velocity and convergence of the density to a traveling profile, completing the proof of \cref{thm:main}. Finally, \cref{app:periodic-commutator} provides a proof of the periodic endpoint commutator estimate.

\subsection*{Notations.}
We work on $\T^d=(\R/2\pi\Z)^d$ with its geodesic distance
$d_{\T^d}$.
Our Fourier convention is
$\widehat f(\xi)=(2\pi)^{-d}\int_{\T^d}f(x)e^{-ix\cdot\xi}\dd x$,
$\xi\in\Z^d$, and $\mean f:=\widehat f(0)$.
For $s\neq0$, the operator $\Lambda^s=(-\Delta)^{s/2}$ is defined by
\begin{equation}
  \widehat{\Lambda^sf}(\xi)=
  \begin{cases}
    |\xi|^s\widehat f(\xi),&\xi\neq0,\\
    0,&\xi=0.
  \end{cases}
  \label{eq:negative-Lambda-definition}
\end{equation}
In particular, negative powers and the Riesz transforms
$R_k:=\partial_k\Lambda^{-1}$ annihilate constants.

Identifying $f$ with its periodic extension to $\R^d$, we write
\begin{equation}
  \Lambda^\alpha f(x)
  =c_{d,\alpha}\operatorname{p.v.}\int_{\R^d}
    \frac{f(x)-f(y)}{|x-y|^{d+\alpha}}\dd y,
  \qquad 0<\alpha<2,
  \label{eq:fractional-laplacian}
\end{equation}
where $c_{d,\alpha}>0$ is normalized so that $\Lambda^\alpha$ has
Fourier symbol $|\xi|^\alpha$.
For $0<\alpha<1$ and $f\in C^1(\T^d)$, the integral is absolutely
convergent.  All integrals over $\R^d$ involving periodic functions
are understood in this sense.
The phrase \emph{fixed periodic normalization} refers to these
choices of torus, Fourier transform, and kernel.  Constants arising
from this normalization depend only on $d,\alpha$ and may change
from line to line.

\section{Preliminaries}
\label{sec:setup}
\label{sec:preliminaries}

\subsection{Local well-posedness theory and a priori bounds}
We collect the known results used below, following \cite[Section~2]{li2024global}, and state them without proof.

First, we state the local well-posedness theory and regularity criterion for smooth solutions of the system \eqref{eq:main}.
\begin{proposition}\label{prop:LWP}
Consider the system \eqref{eq:main} with initial data satisfying \eqref{eq:data}.  Then there exist a
maximal time $T_*\in(0,\infty]$ and a unique non-vacuum classical
solution $(\rho,u)$ on $[0,T_*)$ such that
\begin{equation}
  \begin{aligned}
    \rho
    &\in
    C_{\mathrm w}\bigl([0,T_*);H^{m+\alpha}(\T^d)\bigr),
    \\
    u
    &\in
    C_{\mathrm w}\bigl([0,T_*);H^{m+1}(\T^d)\bigr)
    \cap
    L^2_{\mathrm{loc}}
    \bigl([0,T_*);\dot H^{m+1+\alpha/2}(\T^d)\bigr).
  \end{aligned}
  \label{eq:baseline-solution-class}
\end{equation}
Here $C_{\mathrm w}$ denotes weak continuity in time, and $T_*$ is the
maximal existence time in the solution class
\eqref{eq:baseline-solution-class}.
Moreover, if $T_*<\infty$, then
\begin{equation}
  \sup_{0\leq t<T_*}
  \left(
    \norm{\nabla\rho(\cdot,t)}{L^\infty(\T^d)}
    +
    \norm{\nabla u(\cdot,t)}{L^\infty(\T^d)}
  \right)
  =
  \infty.
  \label{eq:local-blowup}
\end{equation}
\end{proposition}
The local theory and continuation criterion are given in
\cite[Theorem~1.1]{lear2021unidirectional}.

Next, we state some a priori bounds.
Denote the auxiliary quantity
\begin{equation*}
  G:=\partial_1u-\Lambda^\alpha\rho.
\end{equation*}
It satisfies the following conservation law:
\begin{equation*}
  \partial_tG+\partial_1(Gu)=0.
\end{equation*}
Define $F:=G/\rho$. Then $F$ satisfies the transport equation
\begin{equation}\label{eq:F-transport}
  (\partial_t+u\partial_1)F=0.
\end{equation}
Consequently, we have 
\begin{equation*}
  \norm{F(\cdot,t)}{L^\infty}=\norm{F_0}{L^\infty},
\end{equation*}
where $F_0:=F(\cdot,0)$.

We have the following a priori bounds on the density: there exist constants
$0<\underline\rho\leq\overline\rho<\infty$,
depending only on the initial data, $d,\alpha$, and the periodic
normalization, such that
\begin{equation}
  \underline\rho \leq \rho(x,t) \leq \overline\rho,
  \qquad (x,t)\in\T^d\times[0,T_*).
  \label{eq:uniform-density-bounds}
\end{equation}
See, e.g., \cite[Proposition~2.2]{li2024global}.
This then implies 
\begin{equation}\label{eq:G-Linf}
  \norm{G(\cdot,t)}{L^\infty} = \norm{\rho(\cdot,t) F(\cdot,t)}{L^\infty} \leq \overline\rho \norm{F_0}{L^\infty}.
\end{equation}

Finally, the velocity satisfies the maximum principle
$\norm{u(t)}{L^\infty}\leq\norm{u_0}{L^\infty}$ and the following
flocking estimate; see \cite[Lemma~2.3]{li2024global}
and \cite{tadmor2014critical}.  Define
\[
 V(t):=\sup_{x,y\in\T^d}|u(x,t)-u(y,t)|,
 \qquad V_0:=V(0).
\]
Then there exists $\kappa>0$, depending only on $d,\alpha$ and the
initial mass, such that
\begin{equation}\label{eq:velocity-alignment}
  V(t)\leq V_0 e^{-\kappa t},\qquad 0\leq t<T_*.
\end{equation}

\subsection{The endpoint commutator}
For $k\in\{1,\ldots,d\}$, write
\[
  \mathcal R_{k,\alpha}:=\partial_k\Lambda^{-\alpha}.
\]
We shall use the following Kato--Ponce type commutator estimate to control the forcing term $f$ in \eqref{eq:H-intro}.
\begin{lemma}
\label{lem:commutator}
Let $0<\alpha<1$ and $1<p<\infty$.  For every
$k\in\{1,\ldots,d\}$ and every $v,h\in C^1(\T^d)$,
\begin{equation}
  \norm{\comm{\mathcal R_{k,\alpha}}{v}h}{L^p(\T^d)}
  \leq
  C_{d,\alpha,p}
  \norm{\nabla\Lambda^{-\alpha}v}{L^p(\T^d)}
  \norm{h}{L^\infty(\T^d)}.
  \label{eq:commutator}
\end{equation}
The constant also depends on the fixed periodic normalization.
\end{lemma}

The estimate \eqref{eq:commutator} is an endpoint commutator estimate in the sense that the second factor is measured in  $L^\infty$. Its Euclidean counterpart on $\R^d$ follows, for example, from \cite[Corollary~1.4(2)]{li2019kato}, with $s=s_1=1-\alpha$ and $s_2=0$. The corresponding estimate on the periodic domain $\T^d$ can be obtained by a standard transference/localization argument; see also \cite{benyi2024fractional} for periodic fractional Leibniz and related commutator estimates. For completeness, we provide a proof in \cref{app:periodic-commutator}.

\section{A priori bounds for the potential gradient}
\label{sec:Gamma-bounds}

We introduce the potential $\phi$ and its gradient $\G$:
\[
  \phi := \Lambda^{-\alpha}(u-\mean u),
  \quad\text{and}\quad
  \G := \nabla\phi = \nabla\Lambda^{-\alpha}\big(u-\mean u\big).
\]
On periodic lifts, with the period rescaled, $\G$ satisfies the same critical scaling as $\rho$:
\[
 \G_\lambda(x,t) = \G(\lambda x, \lambda^\alpha t).
\]
In this section, we establish the $L^\infty$ bound for $\G$. The estimate will be used in the modulus-of-continuity argument in the next section.

Throughout this section, we use the structural relation
\begin{equation}
  \partial_1u=\Lambda^\alpha\rho+G,
  \label{eq:Gamma-structural}
\end{equation}
together with the a priori bounds obtained in \cref{sec:preliminaries},
\begin{equation*}
  0<\underline\rho\leq \rho(x,t)\leq\overline\rho,
  \qquad
  \norm{G(t)}{L^\infty}\leq G_*,
  \qquad
  0\leq t<T_*.
\end{equation*}
Here and below, constants may depend on the fixed periodic
normalization.  We take $G_*:=\overline\rho\norm{F_0}{L^\infty}$,
as provided by \eqref{eq:G-Linf}.

\subsection{The potential-gradient equation}

For $k\in\{1,\ldots,d\}$, denote
\begin{equation*}
  \G :=(\Gamma_1,\ldots,\Gamma_d),
  \quad\text{where}\quad
  \Gamma_k :=\mathcal R_{k,\alpha}u,
  \quad\text{and}\quad
  \mathcal R_{k,\alpha} :=\partial_k\Lambda^{-\alpha}.
\end{equation*}
We first derive the equation satisfied by
$\G$.  By \eqref{eq:Gamma-structural}, the velocity
equation can be rewritten as
\[
  \partial_tu
  =-u\partial_1u-\Lambda^\alpha(\rho u) + u\Lambda^\alpha\rho
  =-\Lambda^\alpha(\rho u)-uG.
\]
Applying $\mathcal R_{k,\alpha}$ and using
$\mathcal R_{k,\alpha}\Lambda^\alpha=\partial_k$, we obtain
\begin{equation*}
  \partial_t\Gamma_k
  =-\partial_k(\rho u)-\mathcal R_{k,\alpha}(uG).
\end{equation*}
Moreover,
\[
  \Lambda^\alpha\Gamma_k=\partial_ku,
  \quad\text{and}\quad
  \partial_1\Gamma_k =\partial_k\rho+\mathcal R_{k,\alpha}G.
\]
Consequently,
\begin{align*}
  \partial_t\Gamma_k +u\partial_1\Gamma_k +\rho\Lambda^\alpha\Gamma_k
  = &
  -\partial_k(\rho u) -\mathcal R_{k,\alpha}(uG) +u\bigl(\partial_k\rho+\mathcal R_{k,\alpha}G\bigr) +\rho\partial_ku\\
  = &
  u\mathcal R_{k,\alpha}G -\mathcal R_{k,\alpha}(uG)
  = -\comm{\mathcal R_{k,\alpha}}{u}G.
\end{align*}
Thus, $\G$ satisfies the equation
\begin{equation}\label{eq:Gamma-equation}
  \partial_t\G +u\partial_1\G +\rho\Lambda^\alpha\G =f, 
\end{equation}
where the forcing term has a commutator form
\begin{equation*}
    f=(f_1,\ldots,f_d),\qquad
	f_k :=-\comm{\mathcal R_{k,\alpha}}{u}G.
\end{equation*}

\subsection{The \texorpdfstring{$L^p$}{L-p} bound}

We first obtain an $L^p$ bound for $\G$ for every finite $p$.

\begin{proposition}\label{prop:Gamma-Lp}
Let $2\leq p<\infty$.  Then
\begin{equation}
  \norm{\G(t)}{L^p}
  \leq
  e^{C_pG_*t}
  \norm{\G_0}{L^p},
  \qquad
  0\leq t<T_*.
  \label{eq:Gamma-Lp}
\end{equation}
\end{proposition}

\begin{proof}
Multiplying \eqref{eq:Gamma-equation} by $p|\G|^{p-2}\G$ and integrating over $\T^d$, we obtain
\begin{align}\label{eq:Lpest}
  & \frac{\dd}{\dd t}\int_{\T^d}|\G|^p\dd x = 
  \int_{\T^d}(\partial_1u)\,|\G|^p\dd x
  - p\int_{\T^d}\rho\,|\G|^{p-2}\G\cdot\Lambda^\alpha\G\dd x
  + p\int_{\T^d}f\cdot |\G|^{p-2}\G\dd x\notag\\
  & \qquad = 
  \int_{\T^d}\Lambda^\alpha\rho\,|\G|^p\dd x
  + \int_{\T^d}G\,|\G|^p\dd x
  - p\int_{\T^d}\rho\,|\G|^{p-2}\G\cdot\Lambda^\alpha\G\dd x
  + p\int_{\T^d}f\cdot |\G|^{p-2}\G\dd x,
\end{align}
where we have used the relation \eqref{eq:Gamma-structural}.

Applying the vector-valued C\'ordoba--C\'ordoba inequality \cite{cordoba2003pointwise} 
to the convex function $s\mapsto |s|^p$ with $s=\G$, we have
\begin{equation}\label{eq:CC}
  \Lambda^\alpha(|\G|^p) \leq p|\G|^{p-2}\G \cdot\Lambda^\alpha\G.
\end{equation}
Therefore,
\[
 \int_{\T^d}\Lambda^\alpha\rho\,|\G|^p\dd x 
 = \int_{\T^d}\rho\,\Lambda^\alpha(|\G|^p)\dd x 
 \le p\int_{\T^d}\rho\,|\G|^{p-2}\G \cdot\Lambda^\alpha\G\dd x.
\]

Noting that
\[
 \norm{|\G|^{p-2}\G}{L^{\frac{p}{p-1}}} = \left(\int_{\T^d} |\Gamma|^{(p-1)\cdot\frac{p}{p-1}}\dd x\right)^{\frac{p-1}{p}}
 = \norm{\G}{L^p}^{p-1},
\]
we apply H\"older's inequality to the remaining two terms in \eqref{eq:Lpest} and obtain
\[
  \frac{\dd}{\dd t}\norm{\G}{L^p}^p
  \leq \norm G{L^\infty} \norm{\G}{L^p}^p + p \norm f{L^p} \norm{\G}{L^p}^{p-1}.
\]

Finally, for the forcing term, we apply commutator estimate \eqref{eq:commutator} and obtain
\begin{equation}\label{eq:fcommutator}
 \norm{f_k}{L^p} = \norm{\comm{\mathcal R_{k,\alpha}}{u}G}{L^p}
 \le C_{d,\alpha,p} \norm{\nabla\Lambda^{-\alpha}u}{L^p} \norm{G}{L^\infty} 
 \le C_{d,\alpha,p} G_* \norm{\G}{L^p}.
\end{equation}

Applying all the estimates to \eqref{eq:Lpest}, we find
\[
  \frac{\dd}{\dd t}
  \norm{\G}{L^p}^p
  \leq
  C_pG_*\norm{\G}{L^p}^p.
\]
Gr\"onwall's inequality gives \eqref{eq:Gamma-Lp}.
\end{proof}

As a direct consequence of \eqref{eq:Gamma-Lp} and \eqref{eq:fcommutator}, we have the bound
\begin{equation}\label{eq:Gamma-f-Lp}
  \norm{f(t)}{L^p}
  \leq C_pG_*e^{C_pG_*t} \norm{\G_0}{L^p}.
\end{equation}

\subsection{The \texorpdfstring{$L^\infty$}{L-infinity} bound}

Proposition ~\ref{prop:Gamma-Lp} does not directly extend to the case when $p=\infty$. In particular, the estimate in \eqref{eq:fcommutator} ultimately uses the $L^p$-boundedness of the periodic Riesz transforms, while the Riesz transforms do not map $L^\infty$ to $L^\infty$. Thus we do not have a direct estimate of the form
\[
  \norm f{L^\infty} \leq C G_* \norm{\G}{L^\infty}.
\]

To obtain an $L^\infty$ bound, we first use the dissipation to obtain an $L^p\to L^\infty$ smoothing estimate for the homogeneous equation corresponding to \eqref{eq:Gamma-equation}:
\begin{equation}\label{eq:Gamma-homogeneous}
  \partial_tv+u\partial_1v+\rho\Lambda^\alpha v=0,
  \qquad
  v(\cdot,s)=v_s,
\end{equation}
where $v$ is scalar-valued. The estimate will be applied
componentwise to $\G$.

We first state the following maximum principle.
\begin{lemma}
\label{lem:max}
For every scalar solution of \eqref{eq:Gamma-homogeneous} and every
$s\leq t$,
\begin{equation}\label{eq:Gamma-max-principle}
  \norm{v(t)}{L^\infty} \leq\norm{v_s}{L^\infty}.
\end{equation}
\end{lemma}

\begin{proof}
Suppose that $v$
attains a positive spatial maximum at $x_t$. Then
\[
  \partial_1v(x_t,t)=0,
  \qquad
  \Lambda^\alpha v(x_t,t)\geq0.
\]
Since $\rho(x_t,t)>0$,
\eqref{eq:Gamma-homogeneous} implies
\[
  \frac{\dd^+}{\dd t}
  \max_x v(x,t)\leq0.
\]
Applying the same argument to $-v$, we obtain
\eqref{eq:Gamma-max-principle}.
\end{proof}

Next, we derive the $L^p\to L^\infty$ smoothing estimate, using the Moser iteration. 
\begin{proposition}
\label{prop:Gamma-smoothing}
Let $2\leq p<\infty$. Then every smooth solution of
\eqref{eq:Gamma-homogeneous} satisfies
\begin{equation} \label{eq:Gamma-smoothing}
  \norm{v(t)}{L^\infty}
  \leq C\left(1+(t-s)^{-d/(\alpha p)}\right)\norm{v_s}{L^p},
  \qquad
  0\leq s<t<T_*.
\end{equation}
Here $C$ depends only on $d,\alpha,p,\underline\rho,G_*$, and the fixed periodic normalization.
\end{proposition}

\begin{proof}
We first derive an $L^q$ energy inequality which is uniform in $q$. Fix $q\geq2$. The same calculation as \eqref{eq:Lpest} yields
\begin{align*}
 \frac{\dd}{\dd t}\norm v{L^q}^q 
 & = \int_{\T^d}\rho \Big(\Lambda^\alpha(|v|^q) - q|v|^{q-2}v \Lambda^\alpha v\Big)\dd x +\int_{\T^d} G |v|^q\dd x\\
 & \le
 \rhom\int_{\T^d} \Big(\Lambda^\alpha(|v|^q) - q|v|^{q-2}v \Lambda^\alpha v\Big)\dd x + G_*\norm v{L^q}^q,
\end{align*}
where we have used the C\'ordoba--C\'ordoba inequality \eqref{eq:CC}. Then we apply the periodic version of the Stroock--Varopoulos inequality \cite[Lemma~5.1]{depablo2012general}:
\begin{equation*}
  \int_{\T^d} q|v|^{q-2}v\Lambda^\alpha v\dd x
  \geq \frac{4(q-1)}{q} \norm{\Lambda^{\alpha/2}(|v|^{q/2})}{L^2}^2
  \geq 2\norm{\Lambda^{\alpha/2}(|v|^{q/2})}{L^2}^2.
\end{equation*}
It yields
\begin{equation} \label{eq:Gamma-energy-bound}
  \frac{\dd}{\dd t}\norm v{L^q}^q + 2\rhom \norm{\Lambda^{\alpha/2}(|v|^{q/2})}{L^2}^2
  \leq G_*\norm v{L^q}^q.
\end{equation}
In particular, discarding the dissipative term gives
\begin{equation}  \label{eq:Gamma-Lq-growth}
  \norm{v(t)}{L^q}^q
  \leq e^{G_*(t-\tau)} \norm{v(\tau)}{L^q}^q,
  \qquad
  s\leq\tau\leq t.
\end{equation}

Next, we use the dissipation to increase the integrability exponent. Set
\[
  \chi=\frac{d}{d-\alpha}>1.
\]
The periodic fractional Sobolev inequality gives
\[
 \norm h{L^{2\chi}}^2
 \leq C\left(\norm{\Lambda^{\alpha/2}h}{L^2}^2+\norm h{L^2}^2\right).
\]
Applying this estimate to $h=|v|^{q/2}$, we obtain
\[
  \norm v{L^{q\chi}}^q
  \leq C \norm{ \Lambda^{\alpha/2}(|v|^{q/2}) }{L^2}^2 + C\norm v{L^q}^q.
\]
Combining this inequality with \eqref{eq:Gamma-energy-bound} yields
\begin{equation} \label{eq:Gamma-energy-bound2}
  \frac{\dd}{\dd t}\norm v{L^q}^q + c \norm{v}{L^{q\chi}}^q
  \leq C\norm v{L^q}^q.
\end{equation}
Let
\[
  \tau<\sigma,
  \qquad
  0<\sigma-\tau\leq1.
\]
Integrating \eqref{eq:Gamma-energy-bound2} on $[\tau,\sigma]$, we get
\begin{equation}\label{eq:iteration-step}
  \int_\tau^\sigma\norm{v(r)}{L^{q\chi}}^q\dd r
  \leq C \norm{v(\tau)}{L^q}^q,
\end{equation}
where $C$ is independent of $q\geq2$.

We now perform the Moser iteration. Fix $0<t-s\leq1$. Define
\[
  q_n=p\chi^n,
  \qquad
  t_n=s+(t-s)(1-2^{-n}),
  \qquad
  n=0,1,2,\ldots,
\]
so that
\[
  t_0=s,\qquad t_n\uparrow t,
  \qquad
  t_{n+1}-t_n=2^{-n-1}(t-s).
\]
Applying \eqref{eq:iteration-step} with
\[
  q=q_n,\qquad \tau=t_n,\qquad \sigma=t_{n+1},
\]
and using $q_{n+1}=q_n\chi$, we obtain
\[
  \int_{t_n}^{t_{n+1}} \norm{v(r)}{L^{q_{n+1}}}^{q_n}\dd r
  \leq C\norm{v(t_n)}{L^{q_n}}^{q_n}.
\]
Hence there exists $r_n\in[t_n,t_{n+1}]$ such that
\[
  \norm{v(r_n)}{L^{q_{n+1}}}
  = \left(\frac{1}{t_{n+1}-t_n}\int_{t_n}^{t_{n+1}} \norm{v(r)}{L^{q_{n+1}}}^{q_n}\dd r\right)^{1/q_n}
  \leq \left(\frac{C2^{n+1}}{t-s}\right)^{1/q_n}\norm{v(t_n)}{L^{q_n}}.
\]
Using \eqref{eq:Gamma-Lq-growth} at exponent $q_{n+1}$,
from $r_n$ to $t_{n+1}$, gives
\begin{equation}
  \norm{v(t_{n+1})}{L^{q_{n+1}}}
  \leq e^{G_*/q_{n+1}}\norm{v(r_n)}{L^{q_{n+1}}}
  \leq \left(\frac{C2^{n+1}}{t-s}\right)^{1/q_n} e^{G_*/q_{n+1}} \norm{v(t_n)}{L^{q_n}}.
  \label{eq:Gamma-Moser-iteration}
\end{equation}
Iterating \eqref{eq:Gamma-Moser-iteration} from $n=0$ to $N-1$, we find
\begin{equation}\label{eq:Gamma-Moser-product}
  \norm{v(t_N)}{L^{q_N}}
  \leq \prod_{n=0}^{N-1} \left(C^{1/q_n} 2^{(n+1)/q_n} e^{G_*/q_{n+1}}\right)
  \cdot (t-s)^{-\sum_{n=0}^{N-1}1/q_n} \norm{v_s}{L^p}.
\end{equation}
Since
\[
  \sum_{n=0}^\infty\frac1{q_n}
  = \frac1p \sum_{n=0}^\infty\chi^{-n}
  = \frac d{\alpha p},
  \quad\text{and}\quad
  \sum_{n=0}^\infty \frac{n+1}{q_n}<\infty,
\]
the product in \eqref{eq:Gamma-Moser-product} converges.
Letting $N\to\infty$, we conclude that
\begin{equation*}
  \norm{v(t)}{L^\infty}
  \leq
  C
  (t-s)^{-d/(\alpha p)}
  \norm{v_s}{L^p},
  \qquad
  0<t-s\leq1.
\end{equation*}

For $t>s+1$, we apply the maximum principle \eqref{eq:Gamma-max-principle} and get
\[
  \norm{v(t)}{L^\infty} \leq \norm{v(s+1)}{L^\infty} \leq C\|v_s\|_{L^p}.
\]
This finishes the proof.
\end{proof}

We are ready to obtain the $L^\infty$ bound on $\G$. We apply Duhamel's formula and use the smoothing estimate to control the forcing term.
\begin{proposition}\label{prop:Gamma-Linfty}
For every $0<T<T_*$, there exists $C_T>0$, depending only on $T$, $d$, $\alpha$, $\underline\rho$, $G_*$, and the periodic normalization, such that
\begin{equation}\label{eq:Gamma-Linfty}
  \sup_{0\leq t\leq T} \norm{\G(t)}{L^\infty} 
  \leq C_T \norm{\G_0}{L^\infty}.
\end{equation}
 \end{proposition}

\begin{proof}
Denote $S(t,s)$ the solution operator of the homogeneous equation \eqref{eq:Gamma-homogeneous}, namely
\[
v(t) = S(t,s)\, v_s.
\]
Apply Duhamel's formula for each component of \eqref{eq:Gamma-equation}:
\begin{equation}
  \Gamma_k(t) = S(t,0)\, \Gamma_{k,0} + \int_0^t S(t,s)\, f_k(s)\dd s.
  \label{eq:Gamma-Duhamel}
\end{equation}
For the first term, we simply apply the maximum principle \eqref{eq:Gamma-max-principle} and get
\[
  \norm{S(t,0)\, \Gamma_{k,0}}{L^\infty} \leq \norm{\Gamma_{k,0}}{L^\infty}.
\]
For the second term, fix $p>\max\left\{2,\frac d\alpha\right\}$. The smoothing estimate \eqref{eq:Gamma-smoothing} yields
\begin{align*}
  \norm{\int_0^t S(t,s)\, f_k(s)\dd s}{L^\infty}
  &\leq
  C\int_0^t \left(1+(t-s)^{-\beta}\right) \norm{f_k(s)}{L^p}\dd s,
\end{align*}
where $\beta:=\frac{d}{\alpha p}\in(0,1)$.
Using \eqref{eq:Gamma-f-Lp}, we further estimate
\[
 \norm{\int_0^t S(t,s)\, f_k(s)\dd s}{L^\infty}
 \leq C\norm{\G_0}{L^p}\int_0^t \left(1+(t-s)^{-\beta}\right) e^{Cs}\dd s
 \leq C \norm{\G_0}{L^\infty}  \left(T+\frac{T^{1-\beta}}{1-\beta}\right) e^{CT}.
\]
Combining the two terms yields \eqref{eq:Gamma-Linfty} with $C_T = C \left[1+\left(T+\frac{T^{1-\beta}}{1-\beta}\right)e^{CT}\right]$.
In particular, these constants remain bounded as $T\uparrow T_*<\infty$.
\end{proof}

\subsection{Exponential decay of the potential gradient}
\label{subsec:Gamma-decay}

Using the flocking estimate in \eqref{eq:velocity-alignment},
we can upgrade the preceding bounds for $\G$ to exponential decay.

\begin{proposition}[Decay of the potential gradient]
\label{prop:Gamma-decay}
For every fixed $2\leq p<\infty$, there exists $C_p>0$, depending
only on the initial data, $d,\alpha,p$, and the fixed periodic
normalization, such that
\begin{equation}
  \norm{\G(t)}{L^p}\leq C_pe^{-\kappa t},
  \qquad 0\leq t<T_*.
  \label{eq:Gamma-decay-Lp}
\end{equation}
Moreover, there exists a constant $C_\Gamma>0$, with the same
dependence except for $p$, such that
\begin{equation}
  \norm{\G(t)}{L^\infty}\leq C_\Gamma e^{-\kappa t},
  \qquad 0\leq t<T_*.
  \label{eq:Gamma-decay-Linfty}
\end{equation}
\end{proposition}

\begin{proof}
If $V_0=0$, \eqref{eq:velocity-alignment} implies that $u$ is
spatially constant and $\G\equiv0$, so both conclusions are immediate.
Suppose henceforth that $V_0>0$, and write
\[
  \phi=\Lambda^{-\alpha}u,\qquad \G=\nabla\phi.
\]
The periodic convolution kernel $H_\alpha$ of $\Lambda^{-\alpha}$
is integrable, with a singularity of order $|h|^{\alpha-d}$,
and has zero mean by \eqref{eq:negative-Lambda-definition}.
Hence
\[
  \phi(x,t)=\int_{\T^d}H_\alpha(x-y)
      \bigl(u(y,t)-u(x,t)\bigr)\dd y.
\]
Using \eqref{eq:velocity-alignment}, we obtain
\begin{equation}
  \norm{\phi(t)}{L^\infty}
  \leq \norm{H_\alpha}{L^1}V(t)
  \leq A_0e^{-\kappa t},
  \qquad A_0:=\norm{H_\alpha}{L^1}V_0>0.
  \label{eq:potential-decay}
\end{equation}

To proceed, we refine the estimate \eqref{eq:CC} as follows:
\begin{align*}
  p|\G|^{p-2}\G \cdot\Lambda^\alpha\G - \Lambda^\alpha(|\G|^p)
  & \geq p|\G|^{p-2}\G \cdot\Lambda^\alpha\G - \frac{p}{2}\big(|\G|^2\big)^{\frac{p}{2}-1}\,\Lambda^\alpha(|\G|^2)\\
  & = p|\G|^{p-2}\Big(\G \cdot\Lambda^\alpha\G - \frac12\Lambda^\alpha(|\G|^2)\Big)\geq c_p\frac{|\G|^{p+\alpha}}{\norm\phi{L^\infty}^{\alpha}},
\end{align*}
Here the first inequality is the C\'ordoba--C\'ordoba inequality
for the convex function $s\mapsto s^{p/2}$ on $[0,\infty)$, with
$s=|\G|^2$.  The last inequality follows from
\cite[Theorem~2.5]{constantin2012nonlinear}, whose cutoff proof applies
to periodic extensions using \eqref{eq:fractional-laplacian}.
If $\phi\equiv0$, then $\G\equiv0$ and the quotient is understood as zero.

Inserting the refined estimate into \eqref{eq:Lpest}, and using
$\rho\geq\underline\rho$, \eqref{eq:fcommutator}, and
\eqref{eq:potential-decay}, gives
\[
  \frac{\dd}{\dd t}\norm{\G(t)}{L^p}^p
  +a_pe^{\alpha\kappa t}\norm{\G(t)}{L^p}^{p+\alpha}
  \leq C_pG_*\norm{\G(t)}{L^p}^p,
\]
where $a_p:=c_p\underline\rho A_0^{-\alpha}|\T^d|^{-\alpha/p}>0$.
Multiplying by $e^{p\kappa t}$ yields
\[
  \frac{\dd}{\dd t}\left(e^{\kappa t}\norm{\G(t)}{L^p}\right)^p
  \leq -a_p\left(e^{\kappa t}\norm{\G(t)}{L^p}\right)^{p+\alpha}
  + (C_pG_*+p\kappa) \left(e^{\kappa t}\norm{\G(t)}{L^p}\right)^p.
\]
Comparison with a constant supersolution gives
\[
  e^{\kappa t}\norm{\G(t)}{L^p}
  \leq \max\left\{\norm{\G_0}{L^p},
       \left(\frac{C_pG_*+p\kappa}{a_p}\right)^{1/\alpha}\right\},
\]
which proves \eqref{eq:Gamma-decay-Lp}.

To obtain the endpoint estimate, fix once and for all
$p>\max\{2,d/\alpha\}$, and put
$\beta=d/(\alpha p)<1$.  By \eqref{eq:fcommutator} and
\eqref{eq:Gamma-decay-Lp},
$\norm{f(t)}{L^p}\leq Ce^{-\kappa t}$.
For $1\leq t<T_*$, apply the componentwise Duhamel formula
\eqref{eq:Gamma-Duhamel} on $[t-1,t]$ and the smoothing estimate
\eqref{eq:Gamma-smoothing} to obtain
\begin{align*}
  \norm{\G(t)}{L^\infty}
  &\leq C\norm{\G(t-1)}{L^p}
       +C\int_{t-1}^t\bigl(1+(t-s)^{-\beta}\bigr)
                    \norm{f(s)}{L^p}\dd s\\
  &\leq Ce^{-\kappa(t-1)}
       +Ce^{-\kappa t}\int_0^1(1+r^{-\beta})e^{\kappa r}\dd r
   \leq Ce^{-\kappa t}.
\end{align*}
For $0\leq t<\min\{1,T_*\}$, applying \eqref{eq:Gamma-Linfty} and increasing the constant by a factor $e^{\kappa}$ gives
\[
  \norm{\G(t)}{L^\infty} \leq C \left[1+\left(t+\frac{t^{1-\beta}}{1-\beta}\right)e^{Ct}\right] \norm{\G_0}{L^\infty}\leq C \left[1+\frac{2-\beta}{1-\beta}e^C\right] \norm{\G_0}{L^\infty}e^{\kappa} e^{-\kappa t}.
\]
This proves \eqref{eq:Gamma-decay-Linfty} on the full existence interval.
\end{proof}

\section{Propagation of the moduli of continuity}
\label{sec:MOC-Gamma}

In this section, we propagate simultaneously a time-independent modulus of
continuity for $\rho$ and a time-dependent modulus of continuity for $\G$.
These quantities have invariant amplitude under the critical scaling
(on periodic lifts, with the period rescaled)
\[
    \rho(x,t)\rightsquigarrow\rho(\lambda x,\lambda^\alpha t),
    \qquad
    \G(x,t)\rightsquigarrow\G(\lambda x,\lambda^\alpha t).
\]
The density bounds \eqref{eq:uniform-density-bounds} and
\cref{prop:Gamma-decay} give, for every $0\leq t<T_*$,
\begin{align}
    \sup_{x,y\in\T^d}|\rho(x,t)-\rho(y,t)|
    &\leq\overline\rho,
    \label{eq:rho-MOC-amplitude}\\
    \sup_{x,y\in\T^d}|\G(x,t)-\G(y,t)|
    &\leq M_\Gamma e^{-\kappa t},
    \qquad M_\Gamma:=2C_\Gamma.
    \label{eq:Gamma-MOC-amplitude}
\end{align}
Here $\kappa>0$ is the rate in \eqref{eq:velocity-alignment}, and
$M_\Gamma$ is independent of the existence interval.

We follow the framework of \cite[Section~3]{li2024global}:
first choose the two moduli, describe their breakthrough scenarios,
and then show that the time derivatives have the required signs at a
putative contact.  The decay estimate \eqref{eq:Gamma-MOC-amplitude}
plays the role of the velocity-alignment estimate in that argument.
The principal additional estimate concerns the commutator forcing
in the equation for $\G$.  Throughout this section, constants denoted
by $C$ may depend on $d,\alpha$, the initial data, and $\kappa$, but not on $\delta,\lambda$, or time; a subscript
records any additional dependence.

For points on $\T^d$, write $|x-y|:=d_{\T^d}(x,y)$, and use shortest
lifts when taking their difference in $\R^d$.
All Euclidean integrals below act on periodic extensions.

\subsection{The moduli of continuity and breakthrough scenarios}
\label{subsec:MOC-framework}

Recall the family of moduli used in
\cite[Section~3]{li2024global}:
\begin{equation}
\omega(\xi) = \omega_{\lambda}^{\delta}(\xi) :=
\begin{cases}
    \delta\lambda^{-1}\xi
    -\dfrac14\delta\lambda^{-1-\frac\alpha2}\xi^{1+\frac\alpha2},
    & 0\leq\xi\leq\lambda,\\[6pt]
    \dfrac34\delta
    +\dfrac{\delta}{2}\log\dfrac{\xi}{\lambda},
    & \xi>\lambda,
\end{cases}
\label{eq:MOC-family}
\end{equation}
where the parameters $0<\delta<1$ and $0<\lambda<1$ are to be determined. Note that $\omega$ is an increasing and concave function with $\omega(0)=0$, $\omega'(0^+)=\delta\lambda^{-1}$, and $\omega''(0^+)=-\infty$.
Define
\begin{equation}\label{eq:MOC-omega12}
    \omega_1(\xi):=\omega(\xi),
    \qquad
    \omega_2(\xi,t):=e^{-\kappa t}\omega(\xi).
\end{equation}
Our goal is to show that, for all $0\leq t<T_*$,
\begin{equation}\label{eq:MOC-rho-goal}
    |\rho(x,t)-\rho(y,t)|<\omega_1(|x-y|),
\end{equation}
and
\begin{equation}\label{eq:MOC-Gamma-goal}
    |\G(x,t)-\G(y,t)|<\omega_2(|x-y|,t),
\end{equation}
for every $x\neq y\in\T^d$.  Periodic functions are identified with
their periodic extensions when applying the modulus estimates.
Note that $\omega_2$ is chosen to be time-dependent in order to obtain the flocking estimate.

For any fixed $\delta>0$, the initial data are Lipschitz, so choosing
$\lambda>0$ sufficiently small makes both $\rho_0$ and $\G_0$ obey
$\omega$.  We refer to \cite[Lemma~3.1]{li2024global} for the
standard fitting argument; its proof applies to the vector increment
of $\G_0$ without change.  The propagated moduli imply
\begin{equation*}
    \norm{\nabla\rho(t)}{L^\infty}
    \leq\omega_1'(0^+)=\delta\lambda^{-1},
\end{equation*}
and
\begin{equation*}
    [\G(t)]_{\mathrm{Lip}}
    \leq\partial_\xi\omega_2(0^+,t)
    =\delta\lambda^{-1}e^{-\kappa t}.
\end{equation*}
Thus the second modulus also captures the decay of the potential
gradient.

We use the standard breakthrough scenario of the modulus-of-continuity method \cite{kiselev2007global,li2024global}.
Suppose $t=t_1\in(0,T_*)$ is the first time at which either
\eqref{eq:MOC-rho-goal} or \eqref{eq:MOC-Gamma-goal} fails.
Then there exist $x\neq y\in\T^d$, with
\[
    \xi=|x-y|,
\]
such that either
\begin{equation*}
    \rho(x,t_1)-\rho(y,t_1)
    =
    \omega_1(\xi),
\end{equation*}
or
\begin{equation}
    |\G(x,t_1)-\G(y,t_1)| = \omega_2(\xi,t_1).
    \label{eq:MOC-Gamma-break}
\end{equation}
Both modulus inequalities hold non-strictly on $[0,t_1]$ and
strictly for $t<t_1$.

Accordingly, it suffices to prove, in the density scenario,
\begin{equation}
    \left.\partial_t\bigl(\rho(x,t)-\rho(y,t)\bigr)
    \right|_{t=t_1}<0,
    \label{eq:MOC-rho-sign}
\end{equation}
and, in the $\G$ scenario,
\begin{equation}
    \left.\partial_t|\G(x,t)-\G(y,t)|\right|_{t=t_1}
    +\kappa\omega_2(\xi,t_1)<0.
    \label{eq:MOC-Gamma-sign}
\end{equation}
The second condition is equivalent, at contact, to strict negativity
of the derivative of
$|\G(x,t)-\G(y,t)|/\omega_2(\xi,t)$, exactly as in
\cite[Section~3.1]{li2024global}.

The amplitude bounds restrict the possible breakthrough distances.
In the density scenario, \eqref{eq:rho-MOC-amplitude} gives
\begin{equation*}
    0<\xi\leq\Xi_1:=\omega^{-1}(\overline\rho)
    \leq\lambda\max\left\{
        e^{2\overline\rho/\delta-3/2},1
    \right\}.
\end{equation*}
In the $\G$ scenario, the common factor $e^{-\kappa t_1}$ in
\eqref{eq:Gamma-MOC-amplitude} and \eqref{eq:MOC-Gamma-break}
cancels.  Consequently,
\begin{equation*}
    0<\xi\leq\Xi_2:=\omega^{-1}(M_\Gamma)
    \leq\lambda\max\left\{
        e^{2M_\Gamma/\delta-3/2},1
    \right\}.
\end{equation*}
Set
\begin{equation}\label{eq:MOC-E}
    E:=\max\left\{
        e^{2\overline\rho/\delta-3/2},
        e^{2M_\Gamma/\delta-3/2},1
    \right\}.
\end{equation}
Once $\delta$ is fixed, $E$ is independent of $\lambda$ and time.
We shall choose $\lambda$ so small that $E\lambda<1/2$.

We next recall two standard dissipation bounds associated with the moduli \eqref{eq:MOC-family}.
Suppose a function $g$ obeys $\omega$ and $g(x) - g(y) = \omega(\xi)$. Then, we have
\begin{equation}\label{eq:D-bound}
	\Lambda^\alpha g(x) - \Lambda^\alpha g(y) \geq D_\alpha(\xi)>0,
\end{equation}
and
\begin{equation}\label{eq:A-bound}
	-\Lambda^\alpha g(x) \leq A_\alpha(\xi),
\end{equation}
where
\begin{equation*}
    D_\alpha(\xi) :=
    C_1\left(\int_0^{\xi/2}\frac{2\omega(\xi)-\omega(\xi+2\eta)-\omega(\xi-2\eta)}{\eta^{1+\alpha}}\dd\eta
    +\int_{\xi/2}^{\infty}\frac{2\omega(\xi)-\omega(2\eta+\xi)+\omega(2\eta-\xi)}{\eta^{1+\alpha}}\dd\eta\right),
\end{equation*}
and
\begin{equation*}
    A_\alpha(\xi):= c_{d,\alpha}\,\mathrm{p.v.}\int_{\R^d} \frac{\omega(|\xi e_1-z|)-\omega(\xi)}{|z|^{d+\alpha}}\dd z.
\end{equation*}
With the special choice of $\omega$ in \eqref{eq:MOC-family}, \cite[Lemmas~3.4 and~3.5]{li2024global} (with $\mu=\alpha/2$)
give the following bounds
\begin{equation}
    D_\alpha(\xi)\geq c_D
    \begin{cases}
        \delta\lambda^{-1-\frac{\alpha}{2}}\xi^{1-\frac{\alpha}{2}},
        &0<\xi\leq\lambda,\\[4pt]
        \omega(\xi)\xi^{-\alpha},
        &\xi>\lambda,
    \end{cases}
    \label{eq:MOC-D-bound}
\end{equation}
and
\begin{equation}
    A_{\alpha}(\xi)\leq C
    \begin{cases}
        \delta\lambda^{-\frac{\alpha}{2}}\xi^{-\frac{\alpha}{2}},
        &0<\xi\leq\lambda,\\[4pt]
        \delta\xi^{-\alpha},
        &\xi>\lambda.
    \end{cases}
    \label{eq:MOC-A-bound}
\end{equation}

Next, since $\G=\nabla\Lambda^{-\alpha}(u-\mean u)$ obeys $\omega_2$, we derive a modulus $\Omega$ for the velocity field $u$, using the relation
\begin{equation*}
    u-\mean{u} = -\sum_{j=1}^d \partial_j\Lambda^{\alpha-2}\Gamma_j.
\end{equation*}
Indeed, following \cite[Equation~(4.49)]{do2018global}, we have
\begin{equation}
\Omega(\xi,t) = C\left(
        \int_0^\xi\frac{\omega_2(\eta,t)}{\eta^\alpha}\dd \eta
        +\xi\int_\xi^\infty\frac{\omega_2(\eta,t)}{\eta^{1+\alpha}}\dd \eta
    \right).
\label{eq:MOC-u}	
\end{equation}
For the modulus \eqref{eq:MOC-omega12}, direct calculation gives
\begin{equation}
    \Omega(\xi,t) \leq C e^{-\kappa t}
    \begin{cases}
        \delta\lambda^{-\alpha}\xi,
        &0<\xi\leq\lambda,\\[4pt]
        \omega(\xi)\xi^{1-\alpha},
        &\xi>\lambda.
    \end{cases}
    \label{eq:MOC-Omega-bound}
\end{equation}
Letting $\xi\to0$ in \eqref{eq:MOC-Omega-bound} gives
\begin{equation}
  \norm{\nabla u(t)}{L^\infty}
  \leq C\delta\lambda^{-\alpha}e^{-\kappa t},
  \qquad 0\leq t\leq t_1.
  \label{eq:MOC-grad-u}
\end{equation}

Finally, we record the control of the transported quantity $F=G/\rho$.
Set
\begin{equation*}
    H := \frac{\partial_1F}{\rho}.
\end{equation*}
Then, $H$ satisfies the transport equation
\begin{equation*}
    (\partial_t+u\partial_1)H=0.
\end{equation*}
Therefore
\begin{equation*}
    \norm{\partial_1F(t)}{L^\infty}
    \leq
    \overline\rho
    \norm{H_0}{L^\infty}.
\end{equation*}
Differentiating \eqref{eq:F-transport} gives
\[
    (\partial_t+u\partial_1)\nabla F
    =
    -\nabla u\,\partial_1F,
\]
and hence, using \eqref{eq:MOC-grad-u},
\begin{equation}
\begin{split}
    \norm{\nabla F(t)}{L^\infty}
    &\leq\norm{\nabla F_0}{L^\infty}
    +\overline\rho\norm{H_0}{L^\infty}
        \int_0^t\norm{\nabla u(s)}{L^\infty}\dd s\\
    &\leq\norm{\nabla F_0}{L^\infty}
    +C \kappa^{-1}\delta\lambda^{-\alpha},
    \qquad
    0\leq t\leq t_1.
\end{split}
    \label{eq:MOC-grad-F}
\end{equation}
Since $G=F\rho$, the simultaneous modulus assumptions give
\begin{equation}
\begin{split}
    \norm{\nabla G(t)}{L^\infty}
    &\leq
    \overline\rho
    \norm{\nabla F(t)}{L^\infty}
    +
    \norm{F_0}{L^\infty}
    \norm{\nabla\rho(t)}{L^\infty}
    \\
    &\leq
    C\lambda^{-1},
    \qquad
    0\leq t\leq t_1.
\end{split}
\label{eq:MOC-grad-G}
\end{equation}
In the last inequality we used $0<\alpha<1$ and
$0<\delta,\lambda<1$.  Both bounds are independent of the
putative breakthrough time $t_1$; this is where the exponential factor
in \eqref{eq:MOC-grad-u} is essential.

\subsection{Evolution of the modulus of continuity of the density}
\label{subsec:MOC-rho}

The breakthrough estimate below follows the same argument as in \cite{li2024global}.
In the breakthrough estimates below, omitted time arguments are understood
to be $t_1$, and time derivatives are evaluated at $t=t_1$, with $x,y$ fixed.
Starting from the dynamics
\[
    \partial_t\rho = -\rho\Lambda^\alpha\rho - F\rho^2 - u\partial_1\rho,
\]
we decompose
\begin{align*}
    \partial_t\rho(x)-\partial_t\rho(y)
    &=
    -\rho(y)\bigl(\Lambda^\alpha\rho(x)-\Lambda^\alpha\rho(y)\bigr)
    -(\rho(x)-\rho(y))\partial_1u(x)
    -\rho(y)F(x)(\rho(x)-\rho(y))
    \\
    &\quad
    -\rho(y)^2(F(x)-F(y))
    -\bigl((u\partial_1\rho)(x)-(u\partial_1\rho)(y)\bigr).
\end{align*}
Applying \eqref{eq:D-bound}, \eqref{eq:A-bound},
\eqref{eq:MOC-grad-F} and the velocity modulus
\eqref{eq:MOC-u}, we obtain
\begin{align}
    \partial_t\rho(x,t_1) -\partial_t\rho(y,t_1) \le
    &
    -\underline\rho D_\alpha(\xi)+
    \omega(\xi)
    \left(
        A_\alpha(\xi)
        +2\overline\rho
        \norm{F_0}{L^\infty}
    \right)
    +
    C
    \left(
        1+\kappa^{-1}\delta\lambda^{-\alpha}
    \right)\xi
    +
    \Omega(\xi)\omega'(\xi).
\label{eq:MOC-rho-general}
\end{align}

For $0<\xi<\lambda$, we have $\omega(\xi)\leq\delta\xi/\lambda$
and $\omega'(\xi)\leq\delta/\lambda$.  By
\eqref{eq:MOC-Omega-bound}, the transport term satisfies
\[
  \Omega(\xi)\omega'(\xi)
  \leq Ce^{-\kappa t_1}\delta^2\lambda^{-1-\alpha}\xi
  \leq C\delta\bigl(\delta\lambda^{-1-\alpha/2}
                         \xi^{1-\alpha/2}\bigr).
\]
Using \eqref{eq:MOC-D-bound} and \eqref{eq:MOC-A-bound} in
\eqref{eq:MOC-rho-general}, we obtain
\begin{equation}
  \partial_t\rho(x,t_1)-\partial_t\rho(y,t_1)
  \leq\delta\lambda^{-1-\alpha/2}\xi^{1-\alpha/2}
  \Big(-c_D\underline\rho+C\delta+C\big(\lambda^\alpha+\delta^{-1}\lambda^{1+\alpha} +\kappa^{-1}\lambda\big)\Big).
\label{eq:MOC-rho-small}
\end{equation}
For $\lambda<\xi\leq\Xi_1\leq E\lambda$, we use
$\omega'(\xi)=\delta/(2\xi)$ and $\omega(\xi)\geq3\delta/4$.
The transport term now satisfies
\[
  \Omega(\xi)\omega'(\xi)
  \leq Ce^{-\kappa t_1}\delta\omega(\xi)\xi^{-\alpha}
  \leq C\delta\omega(\xi)\xi^{-\alpha}.
\]
The same bounds on $D_\alpha$ and $A_\alpha$ then give
\begin{equation}
  \partial_t\rho(x,t_1)-\partial_t\rho(y,t_1)
  \leq \omega(\xi)\xi^{-\alpha}
  \Big(-c_D\underline\rho+C\delta+C_E\big(\lambda^\alpha+\delta^{-1}\lambda^{1+\alpha}+\kappa^{-1}\lambda\big)\Big).
\label{eq:MOC-rho-large}
\end{equation}
Both \eqref{eq:MOC-rho-small} and \eqref{eq:MOC-rho-large} can be negative once $\delta$ is chosen sufficiently small and then $\lambda$ is chosen sufficiently small. Thus the density breakthrough is ruled out.

\subsection{Evolution of the modulus of continuity of \texorpdfstring{$\G$}{Gamma}}
\label{subsec:MOC-Gamma}

Suppose that \eqref{eq:MOC-Gamma-break} holds at $t=t_1$. Set
\begin{equation*}
    \vartheta:=\frac{\G(x,t_1)-\G(y,t_1)}
                       {|\G(x,t_1)-\G(y,t_1)|},
    \qquad q(z,t):=\vartheta\cdot\G(z,t).
\end{equation*}
We keep $\vartheta$ fixed and suppress the time variable $t_1$ below.
Then $q$ obeys the non-strict modulus $\omega_2$ and
\begin{equation*}
    q(x)-q(y)=\omega_2(\xi).
\end{equation*}
The downward jump of $\omega_2'$ at $\lambda$ excludes a contact at
$\xi=\lambda$ for the smooth function $q$.
Taking the difference of \eqref{eq:Gamma-equation} at $x$ and $y$
and projecting onto $\vartheta$, we obtain
\begin{align*}
    \partial_t|\G(x,t_1)-\G(y,t_1)|
    =&-\rho(y)\bigl(\Lambda^\alpha q(x)-\Lambda^\alpha q(y)\bigr)
        -(\rho(x)-\rho(y))\Lambda^\alpha q(x)\\
       &-\bigl((u\partial_1q)(x)-(u\partial_1q)(y)\bigr)
        +\vartheta\cdot(f(x)-f(y)).
\end{align*}
If $\rho(x)\geq\rho(y)$, we apply \eqref{eq:D-bound} and
\eqref{eq:A-bound} to $e^{\kappa t_1}q$. If $\rho(x)<\rho(y)$,
we instead write the first two terms as
\[
    -\rho(x)\bigl(\Lambda^\alpha q(x)-\Lambda^\alpha q(y)\bigr)
    +(\rho(y)-\rho(x))\Lambda^\alpha q(y),
\]
and use $\Lambda^\alpha q(y)\leq e^{-\kappa t_1}A_\alpha(\xi)$,
which follows by applying \eqref{eq:A-bound} to $-e^{\kappa t_1}q$
at the reversed contact. Thus, in both cases,
\begin{equation}
\begin{split}
    \partial_t|\G(x,t_1)-\G(y,t_1)|
    \leq{}&e^{-\kappa t_1}\Big(-\underline\rho D_\alpha(\xi)+\omega(\xi)\max\{A_\alpha(\xi),0\}+\Omega(\xi)\omega'(\xi)\Big)+|f(x)-f(y)|.
\end{split}
\label{eq:MOC-Gamma-general}
\end{equation}

The following proposition provides a key estimate on the commutator. The proof is postponed to \cref{subsec:MOC-commutator}.
\begin{proposition}
\label{prop:MOC-commutator}
Suppose that the simultaneous modulus inequalities hold non-strictly
on $[0,t_1]$, with $0<\delta,\lambda<1$ and $E\lambda<1/2$.
For $f$ in \eqref{eq:Gamma-equation} and $\xi=|x-y|$, we have
\begin{equation}
    |f(x,t_1)-f(y,t_1)|\leq Ce^{-\kappa t_1}
    \begin{cases}
        \delta\lambda^{-1}\xi,
        &0<\xi<\lambda,\\[4pt]
        \omega(\xi),
        &\lambda<\xi\leq E\lambda.
    \end{cases}
    \label{eq:MOC-comm-bound}
\end{equation}
Here $C$ depends only on the initial data, the structural parameters,
and $E$, and is independent of $\lambda$ and $t_1$.
\end{proposition}

We now apply these estimates in the two regions.

For $0<\xi<\lambda$, \eqref{eq:MOC-Omega-bound} gives
\begin{equation}
    \Omega(\xi)\omega'(\xi)
    \leq Ce^{-\kappa t_1}\delta^2\lambda^{-1-\alpha}\xi
    \leq C\delta^2\lambda^{-1-\alpha/2}\xi^{1-\alpha/2}.
    \label{eq:MOC-Gamma-transport-small}
\end{equation}
The time derivative of the modulus contributes
\begin{equation}
    \kappa\omega_2(\xi,t_1)
    \leq\kappa e^{-\kappa t_1}\delta\lambda^{-1}\xi
    \leq\kappa\lambda^\alpha
           e^{-\kappa t_1}\delta\lambda^{-1-\alpha/2}\xi^{1-\alpha/2}.
    \label{eq:MOC-time-small}
\end{equation}
Applying \eqref{eq:MOC-D-bound}, \eqref{eq:MOC-A-bound}, \eqref{eq:MOC-comm-bound},  \eqref{eq:MOC-Gamma-transport-small}, and \eqref{eq:MOC-time-small} to \eqref{eq:MOC-Gamma-general}, we obtain
\begin{equation}
    \partial_t|\G(x,t_1)-\G(y,t_1)| +\kappa\omega_2(\xi,t_1)
    \leq e^{-\kappa t_1}\delta\lambda^{-1-\alpha/2}\xi^{1-\alpha/2}
       \left(-c_D\underline\rho+C\delta
             +C_E\lambda^\alpha+\kappa\lambda^\alpha\right).
\label{eq:MOC-Gamma-final-small}
\end{equation}

For $\lambda<\xi\leq\Xi_2\leq E\lambda$, we have
\begin{equation}
    \Omega(\xi)\omega'(\xi)
    \leq Ce^{-\kappa t_1}\delta\omega(\xi)\xi^{-\alpha}
    \leq C\delta\omega(\xi)\xi^{-\alpha},
    \label{eq:MOC-Gamma-transport-large}
\end{equation}
and
\begin{equation}
    \kappa\omega_2(\xi,t_1)
    \leq E^\alpha\kappa\lambda^\alpha
               e^{-\kappa t_1}\omega(\xi)\xi^{-\alpha}.
    \label{eq:MOC-time-large}
\end{equation}
Using \eqref{eq:MOC-D-bound}, \eqref{eq:MOC-A-bound},
\eqref{eq:MOC-comm-bound}, \eqref{eq:MOC-Gamma-transport-large}, and \eqref{eq:MOC-time-large} in \eqref{eq:MOC-Gamma-general} gives
\begin{equation}
	\partial_t|\G(x,t_1)-\G(y,t_1)| +\kappa\omega_2(\xi,t_1)
    \leq e^{-\kappa t_1}\omega(\xi)\xi^{-\alpha}
       \left(-c_D\underline\rho+C\delta
             +C_E\lambda^\alpha+C_E\kappa\lambda^\alpha\right).
\label{eq:MOC-Gamma-final-large}
\end{equation}
Both \eqref{eq:MOC-Gamma-final-small} and
\eqref{eq:MOC-Gamma-final-large} are negative once $\delta$ is chosen
sufficiently small and then $\lambda$ is chosen sufficiently small.
Thus \eqref{eq:MOC-Gamma-sign} holds and the $\G$ breakthrough is ruled out.

\subsection{The estimate on the commutator}
\label{subsec:MOC-commutator}

Now, we focus on the estimate of the commutator term $f$, proving \cref{prop:MOC-commutator}. Recall the kernel representation
\begin{equation}
    f_k(x)=-\bigl(\comm{\mathcal R_{k,\alpha}}{u}G\bigr)(x)
    =-\int_{\R^d}J_k(z)\bigl(u(x-z)-u(x)\bigr)G(x-z)\dd z,
    \label{eq:MOC-comm-kernel}
\end{equation}
where $J_k$ is the Euclidean kernel of $\partial_k\Lambda^{-\alpha}$,
with $|J_k(z)|\leq C|z|^{-d-1+\alpha}$ and
$|\nabla J_k(z)|\leq C|z|^{-d-2+\alpha}$, and $u,G$ are extended
periodically to $\R^d$. This representation agrees with the periodic
Fourier multiplier. The integral is absolutely convergent: smoothness
controls the singularity at zero since $\alpha>0$, and boundedness of
the periodic functions controls infinity since $\alpha<1$.

Fix $t=t_1$ and suppress the time variable. For $\xi=|x-y|\leq E\lambda<1/2$,
adding and subtracting a mixed term in \eqref{eq:MOC-comm-kernel} gives
\begin{align*}
    f_k(x) - f_k(y) = & -\int_{\R^d}J_k(z)
       \Big(u(x-z)-u(y-z)-u(x)+u(y)\Big)G(y-z)\dd z\\
    &-\int_{\R^d}J_k(z)
       \bigl(u(x-z)-u(x)\bigr)\bigl(G(x-z)-G(y-z)\bigr)\dd z
    =:I_1 + I_2.
\end{align*}

For $I_1$, we further write $G(y-z)=G(y)+(G(y-z)-G(y))$ and obtain
\begin{align*}
    I_1={}&-G(y)\bigl(\Gamma_k(x)-\Gamma_k(y)\bigr)
    -\int_{\R^d}J_k(z)\Big(u(x-z)-u(y-z)-u(x)+u(y)\Big)\bigl(G(y-z)-G(y)\bigr)\dd z.
\end{align*}
Observe that
\[
    |u(x-z)-u(y-z)-u(x)+u(y)|
    \leq2\min\big\{\Omega(\xi),\Omega(|z|)\big\}.
\]
Moreover, by \eqref{eq:MOC-grad-G} and $\norm{G}{L^\infty}\leq G_*$, we have
\[
    |G(y-z)-G(y)|\leq C\min\{|z|/\lambda,1\}.
\]
Taking absolute values and integrating radially, we obtain
\begin{equation}
    |I_1|\leq G_*\omega_2(\xi,t_1)
       +C\int_0^\infty
          \min\{r/\lambda,1\}\min\{\Omega(r),\Omega(\xi)\}
          r^{\alpha-2}\dd r.
    \label{eq:MOC-comm-first}
\end{equation}

For $0<\xi<\lambda$, splitting the integral in
\eqref{eq:MOC-comm-first} at $\xi$ and $\lambda$, and using \eqref{eq:MOC-Omega-bound} gives
\begin{align*}
    |I_1|
    &\leq Ce^{-\kappa t_1}\delta\lambda^{-1}\xi
       +C\left(
          \lambda^{-1}\int_0^\xi\Omega(r)r^{\alpha-1}\dd r
          +\lambda^{-1}\Omega(\xi)\int_\xi^\lambda r^{\alpha-1}\dd r
          +\Omega(\xi)\int_\lambda^\infty r^{\alpha-2}\dd r
       \right)\\
    &\leq Ce^{-\kappa t_1}\delta\left(
          \lambda^{-1}\xi+\lambda^{-1-\alpha}\xi^{1+\alpha}\right)
     \leq Ce^{-\kappa t_1}\delta\lambda^{-1}\xi.
\end{align*}
For $\lambda<\xi\leq E\lambda$, we instead split
\eqref{eq:MOC-comm-first} at $\lambda$ and $\xi$:
\begin{align*}
    |I_1|
    &\leq G_*\omega_2(\xi,t_1)
       +C\left(
          \lambda^{-1}\int_0^\lambda\Omega(r)r^{\alpha-1}\dd r
          +\int_\lambda^\xi\Omega(r)r^{\alpha-2}\dd r
          +\Omega(\xi)\int_\xi^\infty r^{\alpha-2}\dd r
       \right)\\
    &\leq Ce^{-\kappa t_1}\left(
          \omega(\xi)+\delta+\int_\lambda^\xi\frac{\omega(r)}r\dd r\right)
     \leq C_Ee^{-\kappa t_1}\omega(\xi).
\end{align*}
Here we have used $\omega(\xi)\geq3\delta/4$ and $\xi/\lambda\leq E$.

For $I_2$, we split the integral at $|z|=\lambda$. In the near field,
\eqref{eq:MOC-grad-G} and the velocity modulus give
\[
    \left|\int_{|z|\leq\lambda}J_k(z)
       \bigl(u(x-z)-u(x)\bigr)\bigl(G(x-z)-G(y-z)\bigr)\dd z\right|
    \leq C\lambda^{-1}\xi\int_0^\lambda\Omega(r)r^{\alpha-2}\dd r.
\]
In the far field, we transfer the difference of $G$ onto the kernel
and the velocity increment. Choose shortest lifts of $x,y$. Then
\[
    G(x-z)-G(y-z)
    =-\int_0^1(x-y)\cdot\nabla_zG\bigl(y+s(x-y)-z\bigr)\dd s.
\]
Integrating by parts on $\{|z|>\lambda\}$ yields
\begin{align*}
    &\left|\int_{|z|>\lambda}J_k(z)
       \bigl(u(x-z)-u(x)\bigr)\bigl(G(x-z)-G(y-z)\bigr)\dd z\right|\\
    &\quad\leq G_*\xi\left[
       \int_{|z|=\lambda}|J_k(z)|\,|u(x-z)-u(x)|\dd S(z)
       +\int_{|z|>\lambda}
          \left|\nabla_z\bigl(J_k(z)(u(x-z)-u(x))\bigr)\right|\dd z
       \right]\\
    &\quad\leq CG_*\xi\norm{\nabla u}{L^\infty}
       \left(\lambda^{\alpha-1}+\int_\lambda^\infty r^{\alpha-2}\dd r\right)
     \leq C\xi\lambda^{\alpha-1}\norm{\nabla u}{L^\infty}.
\end{align*}
Here we used $|u(x-z)-u(x)|\leq\norm{\nabla u}{L^\infty}|z|$.
The boundary term at infinity vanishes, since the velocity increment
is bounded by $V(t_1)$ and $\alpha<1$. Combining the two regions gives
\begin{equation}
    |I_2|\leq C\lambda^{-1}\xi\int_0^\lambda\Omega(r)r^{\alpha-2}\dd r
       +C\xi\lambda^{\alpha-1}\norm{\nabla u}{L^\infty}.
    \label{eq:MOC-comm-second}
\end{equation}

For $0<\xi<\lambda$, \eqref{eq:MOC-Omega-bound},
\eqref{eq:MOC-grad-u}, and \eqref{eq:MOC-comm-second} give
\begin{equation}
    |I_2| \leq Ce^{-\kappa t_1}\delta\xi\left(
        \lambda^{-1-\alpha}\int_0^\lambda r^{\alpha-1}\dd r
        +\lambda^{\alpha-1}\lambda^{-\alpha}\right)
    \leq Ce^{-\kappa t_1}\delta\lambda^{-1}\xi.
\label{eq:MOC-I2-bound}
\end{equation}
For $\lambda<\xi\leq E\lambda$, the estimate \eqref{eq:MOC-I2-bound} is valid for every $\xi$ and gives
\[
    |I_2|\leq Ce^{-\kappa t_1}\delta\frac\xi\lambda
       \leq C_Ee^{-\kappa t_1}\omega(\xi).
\]

Combining all the bounds proves \eqref{eq:MOC-comm-bound}, and hence \cref{prop:MOC-commutator}.

\subsection{Closing the simultaneous modulus argument}
\label{subsec:MOC-close}

We choose the parameters in the same order as in
\cite[Section~3]{li2024global}.  First fix $\delta>0$ sufficiently
small that all the $C\delta$ terms in
\eqref{eq:MOC-rho-small}, \eqref{eq:MOC-rho-large},
\eqref{eq:MOC-Gamma-final-small}, and
\eqref{eq:MOC-Gamma-final-large} satisfy
\begin{equation*}
    C\delta\leq\frac14c_D\underline\rho.
\end{equation*}
Then $E$ in \eqref{eq:MOC-E} is fixed.  Choose $\lambda>0$
sufficiently small that
\begin{enumerate}
\item[(i)]
the initial data $\rho_0$ and $\G_0$ obey $\omega_\lambda^\delta$;
\item[(ii)]
$E\lambda<1/2$;
\item[(iii)]
the lower-order density terms satisfy
\[
    C_E\left(
        \lambda^\alpha+\delta^{-1}\lambda^{1+\alpha}
        +\kappa^{-1}\lambda
    \right)
    \leq\frac14c_D\underline\rho;
\]
\item[(iv)]
the commutator and the derivative of the time-dependent modulus satisfy
\[
    C_E\lambda^\alpha
    (1+\kappa)
    \leq\frac14c_D\underline\rho.
\]
\end{enumerate}
Such a choice is possible since $\alpha>0$, with $\delta$ and $E$ already fixed.
All choices depend only on the initial data and the structural
parameters, and are independent of $t_1$ or any finite time horizon.

With these choices, \eqref{eq:MOC-rho-small} and
\eqref{eq:MOC-rho-large} imply \eqref{eq:MOC-rho-sign}, while
\eqref{eq:MOC-Gamma-final-small} and
\eqref{eq:MOC-Gamma-final-large} imply \eqref{eq:MOC-Gamma-sign}.
Neither breakthrough can therefore occur.  We have proved, for every
$0\leq t<T_*$ and $x\neq y$,
\[
    |\rho(x,t)-\rho(y,t)|<\omega_1(|x-y|),
    \qquad
    |\G(x,t)-\G(y,t)|<\omega_2(|x-y|,t).
\]
In particular,
\begin{equation}
    \sup_{0\leq t<T_*}\norm{\nabla\rho(t)}{L^\infty}
    \leq\delta\lambda^{-1},
    \label{eq:MOC-final-rho}
\end{equation}
\begin{equation}
    [\G(t)]_{\mathrm{Lip}}
    \leq\delta\lambda^{-1}e^{-\kappa t},
    \qquad 0\leq t<T_*,
    \label{eq:MOC-final-Gamma}
\end{equation}
and \eqref{eq:MOC-grad-u} yields
\begin{equation}
    \norm{\nabla u(t)}{L^\infty}
    \leq C\delta\lambda^{-\alpha}e^{-\kappa t},
    \qquad 0\leq t<T_*.
    \label{eq:MOC-final-u}
\end{equation}
Since conservation of mass and momentum gives
$\norm{u(t)-\bar u}{L^\infty}\leq V(t)$,
\eqref{eq:velocity-alignment} also gives exponential
decay of $u-\bar u$ in $W^{1,\infty}$, at the same rate $\kappa$.
The uniform bounds also close the continuation argument in
\cref{sec:global}.

\section{Global regularity and asymptotic alignment}
\label{sec:global}

We complete the proof of \cref{thm:main} by verifying the
Beale--Kato--Majda-type continuation criterion in
\cref{prop:LWP}; see \cite[Theorem~1.1]{lear2021unidirectional}. We then prove exponential
convergence of the density to a traveling profile.
\begin{proof}[Proof of \cref{thm:main}]
Let $(\rho,u)$ be the maximal classical solution of
\cref{prop:LWP}, with existence time $T_*$.
Suppose, for contradiction, that $T_*<\infty$.
The coefficient bounds \eqref{eq:uniform-density-bounds} and
\eqref{eq:G-Linf} hold on $[0,T_*)$,
and \cref{prop:Gamma-decay} gives
\[
  \norm{\G(t)}{L^\infty}\leq C_\Gamma e^{-\kappa t},
  \qquad 0\leq t<T_*.
\]
Consequently the parameter choice in \cref{subsec:MOC-close} can
be made once on the full existence interval: first fix $\delta$, then $E$,
and finally a positive $\lambda$.  These choices depend only on the
initial data and the structural bounds, and are independent
of a putative breakthrough time $t_1<T_*$.
The argument of \cref{sec:MOC-Gamma} thus gives
\[
  \sup_{0\leq t<T_*}
  \left(\norm{\nabla\rho(t)}{L^\infty}
       +[\G(t)]_{\mathrm{Lip}}\right)
  \leq 2\delta\lambda^{-1}<\infty.
\]
Combining this with \eqref{eq:MOC-final-u} yields
\[
  \sup_{0\leq t<T_*}
  \left(\norm{\nabla\rho(t)}{L^\infty}
       +\norm{\nabla u(t)}{L^\infty}\right)<\infty.
\]
This contradicts \eqref{eq:local-blowup}.  Hence $T_*=\infty$.

Existence, uniqueness, and the Sobolev regularity
\eqref{eq:global-solution-class} now follow by continuation of the
local solution.  The uniform positive lower and upper density bounds
are given by \eqref{eq:uniform-density-bounds}.  The parameters in
\cref{subsec:MOC-close} are independent of time, so
\eqref{eq:MOC-final-rho} gives a uniform Lipschitz bound for $\rho$,
while \eqref{eq:MOC-final-Gamma} and \eqref{eq:MOC-final-u} give
\[
  [\G(t)]_{\mathrm{Lip}}+\norm{\nabla u(t)}{L^\infty}
  \leq Ce^{-\kappa t},\qquad t\geq0.
\]
Combining the velocity-gradient bound with
$\norm{u(t)-\bar u}{L^\infty}\leq V(t)$ and
\eqref{eq:velocity-alignment} proves \eqref{eq:main-velocity-alignment}.

For the density alignment, we follow Lear and Shvydkoy
\cite[Step~4.4, p.~827]{lear2021unidirectional}. Define the density in the moving frame by
\[
    \widetilde\rho(x,t):=\rho(x+\bar u t e_1,t).
\]
The continuity equation gives
\[
    \partial_t\widetilde\rho(x,t)
    =-\bigl[(u-\bar u)\partial_1\rho+\rho\partial_1u\bigr]
       (x+\bar u t e_1,t).
\]
Using \eqref{eq:uniform-density-bounds}, \eqref{eq:MOC-final-rho},
and \eqref{eq:main-velocity-alignment}, we obtain
\begin{equation*}
    \norm{\partial_t\widetilde\rho(t)}{L^\infty}
    \leq\norm{u(t)-\bar u}{L^\infty}\norm{\nabla\rho(t)}{L^\infty}
       +\norm{\rho(t)}{L^\infty}\norm{\nabla u(t)}{L^\infty}
    \leq Ce^{-\kappa t}.
\end{equation*}
Hence, for $s\geq t\geq0$,
\[
    \norm{\widetilde\rho(s)-\widetilde\rho(t)}{L^\infty}
    \leq C\int_t^s e^{-\kappa r}\dd r
    \leq\frac{C}{\kappa}e^{-\kappa t}.
\]
Thus $\widetilde\rho(t)$ converges uniformly to a profile $\rho_\infty$.
The uniform Lipschitz bound \eqref{eq:MOC-final-rho} passes to the limit,
so $\rho_\infty\in W^{1,\infty}(\T^d)$. The density bounds and conservation
of mass also pass to the limit. Letting $s\to\infty$ in the preceding
estimate and translating back to the original frame proves
\eqref{eq:main-density-alignment}, with the same rate $\kappa$.
\end{proof}

\appendix

\section{The periodic endpoint commutator estimate}
\label{app:periodic-commutator}

We give a short localization argument to deduce
\cref{lem:commutator} from the Euclidean commutator estimate.

\begin{proof}[Proof of \cref{lem:commutator}]
Set $s=1-\alpha\in(0,1)$ and fix $1<p<\infty$.
Write $A=\mathcal R_{k,\alpha}$ and let $A_{\mathrm E}$ be its
Euclidean counterpart, with symbol $i\xi_k|\xi|^{-\alpha}$.
The Euclidean estimate (see, e.g.,
\cite[Corollary~1.4(2)]{li2019kato}) gives
\begin{equation}
\norm{\comm{A_{\mathrm E}}{\tilde v}\tilde h}{L^p(\R^d)}
\leq
C\norm{\Lambda^s_{\mathrm E}\tilde v}{L^p(\R^d)}
\norm{\tilde h}{L^\infty(\R^d)},
\qquad
\tilde v,\tilde h\in C_c^\infty(\R^d),
\label{eq:Euclidean-one-term}
\end{equation}
where $\Lambda^s_{\mathrm E}=(-\Delta_{\R^d})^{s/2}$.

Let $v,h$ be smooth periodic functions, identified with their periodic
extensions. Choose $\chi\in C_c^\infty(\R^d)$ with $0\leq\chi\leq1$,
$\chi=1$ on $[-1,1]^d$, and support in $[-2,2]^d$.
Set $\chi_L(x)=\chi(x/L)$ and $Q_L=[-L,L]^d$, for $L\geq1$.
For $B=A_{\mathrm E}$ or $\Lambda^s_{\mathrm E}$, the difference kernel satisfies
$|K_B(z)|\leq C|z|^{-d-s}$. Hence, for every smooth periodic $g$,
Minkowski's inequality gives
\begin{align}
\norm{\comm{B}{\chi_L}g}{L^p(\R^d)}
&=
\norm{B(\chi_Lg)-\chi_LBg}{L^p(\R^d)}
\leq
C\norm g{L^\infty}
\int_{\R^d}
\frac{\norm{\chi_L-\chi_L(\cdot-z)}{L^p(\R^d)}}
{|z|^{d+s}}\dd z
\notag\\
&\leq
C\norm g{L^\infty}L^{d/p}
\int_0^\infty
\min\{r/L,1\}\frac{\dd r}{r^{1+s}}
\leq
C\norm g{L^\infty}L^{d/p-s}.
\label{eq:large-box-localization-global}
\end{align}
Here $Bg$, for periodic $g$, is defined through the absolutely convergent
difference integral and agrees with the corresponding periodic Fourier
multiplier. The same estimate holds with $\chi_L$ replaced by $\chi_L^2$.

On $\R^d$, we have the identity
\[
\chi_L^2\comm A v h
=
\comm{A_{\mathrm E}}{\chi_Lv}(\chi_Lh)
-\comm{A_{\mathrm E}}{\chi_L^2}(vh)
+\chi_Lv\comm{A_{\mathrm E}}{\chi_L}h.
\]
Since $\chi_L=1$ on $Q_L$ and $0\leq\chi_L\leq1$, the triangle inequality and
\eqref{eq:large-box-localization-global} with $B=A_{\mathrm E}$ give
\begin{align*}
\norm{\comm A v h}{L^p(Q_L)}
&\leq
\norm{\comm{A_{\mathrm E}}{\chi_Lv}(\chi_Lh)}{L^p(\R^d)}
+\norm{\comm{A_{\mathrm E}}{\chi_L^2}(vh)}{L^p(\R^d)}
+\norm v{L^\infty}
\norm{\comm{A_{\mathrm E}}{\chi_L}h}{L^p(\R^d)}
\\
&\leq
\norm{\comm{A_{\mathrm E}}{\chi_Lv}(\chi_Lh)}{L^p(\R^d)}
+
C\norm v{L^\infty}\norm h{L^\infty}L^{d/p-s}.
\end{align*}

For the first commutator, applying \eqref{eq:Euclidean-one-term} with
$\tilde v=\chi_Lv$ and $\tilde h=\chi_Lh$ yields
\[
\norm{\comm{A_{\mathrm E}}{\chi_Lv}(\chi_Lh)}{L^p(\R^d)}
\leq
C\norm{\Lambda^s_{\mathrm E}(\chi_Lv)}{L^p(\R^d)}
\norm h{L^\infty}.
\]
Using \eqref{eq:large-box-localization-global} with
$B=\Lambda^s_{\mathrm E}$, we further obtain
\[
\norm{\Lambda^s_{\mathrm E}(\chi_Lv)}{L^p(\R^d)}
\leq
\norm{\chi_L\Lambda^s_{\mathrm E}v}{L^p(\R^d)}
+
\norm{\comm{\Lambda^s_{\mathrm E}}{\chi_L}v}{L^p(\R^d)}
\leq
CL^{d/p}\norm{\Lambda^sv}{L^p(\T^d)}
+
C\norm v{L^\infty}L^{d/p-s},
\]
where we used the periodicity of $\Lambda^sv$ and the support property
$\operatorname{supp}\chi_L\subset[-2L,2L]^d$.

Combining the preceding estimates gives
\[
\norm{\comm A v h}{L^p(Q_L)}
\leq
CL^{d/p}
\norm{\Lambda^sv}{L^p(\T^d)}
\norm h{L^\infty(\T^d)}
+
C\norm v{L^\infty(\T^d)}
\norm h{L^\infty(\T^d)}
L^{d/p-s}.
\]

Now take $L=N\pi$, $N\in\N$. Since $Q_{N\pi}$ is the union of $N^d$
fundamental $2\pi$-periodic cells,
\[
\norm{\comm A v h}{L^p(Q_{N\pi})}
=
N^{d/p}\norm{\comm A v h}{L^p(\T^d)}.
\]
Dividing the preceding estimate by $N^{d/p}$ and absorbing the fixed
powers of $\pi$ into the constant, we obtain
\[
\norm{\comm A v h}{L^p(\T^d)}
\leq
C\norm{\Lambda^sv}{L^p(\T^d)}
\norm h{L^\infty(\T^d)}
+
CN^{-s}\norm v{L^\infty(\T^d)}
\norm h{L^\infty(\T^d)}.
\]
Letting $N\to\infty$ and using $s=1-\alpha>0$, we conclude that
\[
\norm{\comm{\mathcal R_{k,\alpha}}{v}h}{L^p(\T^d)}
\leq
C\norm{\Lambda^{1-\alpha}v}{L^p(\T^d)}
\norm h{L^\infty(\T^d)}.
\]

Finally, the $L^p$-boundedness of the periodic Riesz transforms implies
\[
\norm{\Lambda^{1-\alpha}v}{L^p}
\leq
C\norm{\nabla\Lambda^{-\alpha}v}{L^p}.
\]
This proves \eqref{eq:commutator} for smooth periodic functions.
The general case $v,h\in C^1(\T^d)$ follows by periodic mollification.
\end{proof}

\bigskip
\subsection*{Declaration on the use of generative AI}
During the preparation of this manuscript, the authors used ChatGPT (OpenAI) for language refinement, LaTeX formatting, and as an auxiliary tool to review logical consistency across proof steps. All mathematical content and arguments were independently analyzed, reconstructed, and verified by the authors, who take full responsibility for the accuracy and originality of the published work.

\bigskip

\bibliographystyle{plain}

\end{document}